\documentclass{amsart}
\usepackage{amsfonts,amssymb,amsmath,enumerate,verbatim,mathtools,tikz,bm,mathrsfs,tikz-cd,hyperref}
\hypersetup{
    colorlinks=true,
    linkcolor=blue,
    filecolor=magenta,      
    urlcolor=cyan,
}
\usepackage[utf8]{inputenc}
\usepackage[english]{babel}
\usetikzlibrary{matrix}
\usetikzlibrary{arrows}

\usepackage[all]{xy}
\SelectTips{cm}{}
\usepackage{stmaryrd}

\DeclareMathOperator{\h}{H}
\newcommand{\n}{\mathfrak{n}}

\newcommand{\Z}{\mathbb{Z}}

\newcommand{\Poly}{\mathcal{S}}

\newcommand{\N}{\mathbb{N}}
\newcommand{\ov}[1]{\overline{#1}}
\newcommand{\vp}{\varphi}

\newcommand{\con}{\subseteq}

\pgfdeclarelayer{bg}    % declare background layer
\pgfsetlayers{bg,main}

\newcommand{\del}{\partial}
\newcommand{\f}{{\bm{f}}}

\DeclareMathOperator{\z}{Z}

\newcommand{\e}{\epsilon}

\DeclareMathOperator{\Ima}{Im}
\DeclareMathOperator{\Ker}{Ker}

\DeclareMathOperator{\hh}{H}

\DeclareMathOperator{\id}{id}

\DeclareMathOperator{\ann}{ann}

\DeclareMathOperator{\Hom}{Hom}
\DeclareMathOperator{\Ext}{Ext}

\DeclareMathOperator{\Tor}{Tor}

\DeclareMathOperator{\zz}{Z}
\DeclareMathOperator{\bb}{B}
\DeclareMathOperator{\Poin}{\mathsf{P}}
\DeclareMathOperator{\Bass}{\mathsf{I}}

\DeclareMathOperator{\pp}{\mathcal{P}}

\newcommand{\xra}{\xrightarrow}

\newtheorem{theorem}{Theorem}[section]
\newtheorem*{Theorem}{Theorem}

\newtheorem{proposition}[theorem]{Proposition}

\newtheorem{lemma}[theorem]{Lemma}
\newtheorem{corollary}[theorem]{Corollary}

\theoremstyle{definition}

\newtheorem{definition}[theorem]{Definition}
\newtheorem{notation}[theorem]{Notation}
\newtheorem{chunk}[theorem]{ }

\theoremstyle{remark}
\newtheorem{remark}[theorem]{Remark}

\newtheorem*{ack}{Acknowledgements}

\newtheorem{Construction}[theorem]{Construction}

\numberwithin{equation}{section}

\begin{document}

\title[Equivariant isomorphisms of Ext and Tor modules]{Equivariant isomorphisms of Ext and Tor modules}

\author[Joshua Pollitz]{Josh Pollitz}
\address{Department of Mathematics,
University of Nebraska, Lincoln, NE 68588, U.S.A.}
\email{jpollitz@huskers.unl.edu}

\date{\today}


\begin{abstract}
In this article we establish equivariant isomorphisms  of Ext and Tor modules over different relative complete intersections. More precisely,  for a commutative  ring $Q$, this paper  investigates how $\Ext_{Q/(\f)}^*(M,N)$ and  $\Tor^{Q/(\f)}_*(M,N)$ change when one varies $\f$ among all Koszul-regular sequences of a fixed length such that $\f M=0$ and $\f N=0$. Of notable interest is how the theory of perturbations is used to establish isomorphisms of certain DG modules. 
\end{abstract}

\maketitle

\section{Introduction}

Fix a commutative  ring $Q$ and a pair of $Q$-modules $M$ and $N$. In this paper, we study the following problem: 
\emph{How do $\Ext_{Q/(\f)}^*(M,N)$ and $\Tor^{Q/(\f)}_*(M,N)$ change as we vary  $\f$ among all $Q$-regular sequences  of a fixed length such that $\f M=\f N=0$? }

This has been studied   when $Q$ is local with residue field $k$, $\f$ is a single $Q$-regular element and $M=k$ (see \cite{VPD}, \cite{SV}, \cite{AvI}, and \cite{Jor}).  The strongest   result in this direction is the following theorem due to  Avramov and Iyengar \cite[2.1(2)]{AvI}: 
\emph{Let $(Q,\n,k)$ be a commutative noetherian local ring and $I$ an ideal of $Q$. If $f$ and $g$ are $Q$-regular elements in $I$ such that $f-g\in \n I$,  then there is an isomorphism of graded $k$-vector spaces  $$\Tor^{Q/(f)}_*(k,N)\cong \Tor^{Q/(g)}_*(k,N)$$ for each complex of $Q/I$-modules $N$.}

One of the  main results of this article, which can be found in Theorem \ref{t2}, is the following  generalization of \cite[2.1(2)]{AvI}:
\begin{Theorem} 
Let $Q$ be a commutative  ring,   $\f=f_1,\ldots, f_n$ and $\f'=f_1'\ldots, f_n'$ be $Q$-regular sequences in an ideal $I$ of $Q$, and $M$ a $Q/I$-module. 
If $f_i-f_i'\in \ann_Q(M) I $ for each $i$, then 
for each complex of $Q/I$-modules $N$  we have isomorphisms of graded $Q/I$-modules:
\begin{enumerate}
\item $\Ext_{Q/(\f)}^*(M,N)\cong \Ext_{Q/(\f')}^*(M,N)$
\item $\Tor^{Q/(\f)}_*(M,N)\cong \Tor^{Q/(\f')}_*(M,N)$ 
\end{enumerate}
\end{Theorem}

In fact, more is shown in Theorem \ref{t2}. Loosely speaking,  the isomorphisms in (1) and (2), above, respect the cohomology operators which were first introduced by Gulliksen in \cite{G} and later studied by Avramov and  Buchweitz \cite{CD2},  Eisenbud \cite{Eis}, and many others.  As a consequence, even of the weaker result displayed above,  if $Q$ is local then the Bass series (or Poincar\'{e} series)   of the pair $(M,N)$ is the same over $Q/(\f)$ as  over $Q/(\f')$ (c.f. \ref{basspoin} and Corollary \ref{cor1} for precise statements). Hence, other homological invariants, like complexity (or Tor-complexity), are the same for the pair $(M,N)$ when computed over the ring  $Q/(\f)$ and computed over  $Q/(\f')$. 

Secondly, it is worth remarking that the  techniques developed, and used,  in this document  differ  from those used in \cite[2.1(2)]{AvI}; these techniques certainly add  value to this paper by providing a  framework to better understand  perturbations of a complex. It is the authors hope that the techniques of Section \ref{sP} will be of particular interest to researchers in various fields, not just commutative algebraist.

In this paper, we also give a   generalization of \cite[2.1(3)]{AvI}. Again, as a consequence of the next theorem below, we obtain information regarding the numerical data of (co)homology modules associated to pairs of modules over a relative complete intersection (i.e., a commutative ring modulo a regular sequence). 
\begin{Theorem}
Let $Q$ be a commutative  ring,   $\f=f_1,\ldots, f_n$ and  be a $Q$-regular sequence in an ideal $I$ of $Q$, and $M$ a $Q/I$-module.  If $(\f)\con I \ann_Q(M)$, we  have isomorphisms of graded $Q/I$-modules:
\begin{enumerate}\setcounter{enumi}{2}
\item $\Ext_{Q/(\f)}^*(M,N)\cong \pp\otimes_{Q} \Ext_Q^*(M,N) $  
\item $\Tor^{Q/(\f)}_*(M,N)\cong \pp^*\otimes_{Q}\Tor^{Q}_*(M,N)$ 
\end{enumerate}
where $\pp=Q/I[\chi_1,\ldots, \chi_n]$, each $\chi_i$ has cohomological degree two and $\pp^*$ is the graded $Q$-linear dual of $\pp$. 

\end{Theorem}

In what follows, we give a brief outline of this paper. In Section 2, we review  notation and tools from DG homological algebra  which will be needed to discuss the content in the rest of the article.  In Section 3, we discuss \emph{universal resolutions}. This is mostly a summary of  work done by Avramov and Buchweitz in  \cite{CD2}. Universal  resolutions allow one to equip projective resolutions of $Q/(\f)$-modules with a structure of  DG $\Poly$-module where $\Poly$ is a polynomial ring with $n$ variables of cohomological degree two. We leverage this extra structure to obtain a finer result than the theorem mentioned above (c.f. Section 5). 

Section 4 contains most of the new ideas in this article. In \cite[2.1]{AvI}, the authors exploit that the minimal free resolution of the residue field over a local ring has a system of divided powers. Since the main result in this article is for arbitrary $Q/(\f)$-modules, we  do not have access to this tool in such a general situation. This is one of the  major differences in the proof of  \cite[2.1]{AvI} and Theorems \ref{t2} and \ref{t9}.  Instead, Section 4  examines when two perturbations yield isomorphic DG modules. The main result in this section is Theorem \ref{t1}.  Finally, we apply the results in the previous sections  in Section 5.

\begin{ack} I would  like to thank Benjamin Briggs who gave a careful reading of and  comments on a preliminary draft of this article. 
I would also like to thank Luchezar Avramov and Mark  Walker for several discussions regarding this work.\end{ack}

\section{Homological Preliminaries}

Many  results in this article depend on several homological constructions. In this section we set terminology and  conventions, and list some basic properties regarding these constructions. This section can be skipped and referred to as needed. See \cite{IFR}, \cite{AFH} or \cite{FHT} as references. 

Fix a commutative  ring $Q$. Let $A=\{A_i\}_{i\in \Z}$ denote a DG $Q$-algebra. In this article, we will always assume that $A$ is graded-commutative. By a DG $A$-module, we mean a \emph{left} DG $A$-module.

\begin{chunk}Let $M$ be a complex of $Q$-modules. The differential of  $M$ is denoted by  $\del^M$. 
The \emph{boundaries} and \emph{cycles of $M$} are denoted by  $\bb(M):=\{\Ima \del^M_{i+1}\}_{i\in \Z}$ and $\z(M):=\{ \Ker\del^M_i\}_{i\in \Z}$, respectively. 
The \emph{homology of $M$} is defined to be  $$\h(M):=\z(M)/\bb(M)=\{\h_i(M)\}_{i\in \Z}$$ which is a graded $Q$-module.
\end{chunk}

\begin{chunk} Let $M$ be a DG $A$-module. 
 We let $M^\natural$ denote the underlying graded $Q$-module. Note that  $A^\natural$ is a graded $Q$-algebra and $M^\natural$ is a graded $A^\natural$-module.
 Also, $\h(M)$ is a graded $\h(A)$-module. 
\end{chunk}

\begin{chunk}We say that  
 $\alpha: M\to N$ is  a degree $d$-map from $M$ to $N$ is a family of $Q$-linear maps $\alpha=\{\alpha_i:M_i\to N_{i+d}\}_{i\in \Z}$ such that $$\alpha(am)=(-1)^{d|a|}a\alpha(m)$$ for all $a\in A$ and $m\in M$.  
\end{chunk}

\begin{chunk}
We define $\Hom_A(M,N)$ to be the DG $A$-module $\Hom_A(M,N)=\{\Hom_A(M,N)_i\}_{i\in \Z}$ determined by  \begin{align*}\Hom_A(M,N)_d:&=\{\alpha: M \to N: \alpha\text{ is a degree }d\text{ map}\}, \\ \del^{\Hom_A(M,N)}:&=\Hom(M,\del^N)-\Hom(\del^M,N),\text{ and }\\ a\cdot \alpha:&=a\alpha(-)=(-1)^{d|a|}\alpha(a\cdot -).\end{align*} We remark that $\Hom_A(M,N)$ is a subcomplex of $\Hom_Q(M,N)$. 
\end{chunk}

\begin{chunk}
Let $\alpha\in \Hom_A(M,N)_0$. We say that $\alpha$ is a morphism of DG $A$-modules if $\alpha\in \z_0(\Hom_A(M,N)).$ Equivalently,   $$\alpha(am)=a\alpha(m)$$ for all $a\in A$ and $m\in M$. 
\end{chunk}

\begin{chunk}
Let $M$ and $N$ be DG $A$-modules. We say that degree $d$ maps $\alpha$ and $\beta$ from $M$ to $N$ are homotopic, denoted $\alpha\sim \beta$, if $\alpha-\beta \in \bb_d(\Hom_A(M,N)).$  That is, there exists $\tau\in \Hom_A(M,N)_{d+1}$ such that $$\del^N\tau-(-1)^{d+1}\tau \del^N=\alpha-\beta.$$A morphism of DG $A$-modules $\alpha: M\to N$ is a homotopy equivalence if there exists a morphism of DG $A$-modules $\beta: N\to M$ such that $$\beta\alpha\sim \id^M\text{ and } \alpha\beta \sim \id^N.$$ A morphism of $A$-modules $\alpha: M\to N$ is a quasi-isomorphism if $\h(\alpha): \h(M)\to \h(N)$ is an isomorphism of graded $\h(A)$-modules. 
\end{chunk}

\begin{chunk}\label{semiproj}
A DG $A$-module $P$ is  \emph{semiprojective} if for every morphism of DG $A$-modules $\alpha: P\to N$ and each surjective quasi-isomorphism of DG $A$-modules $\gamma: M\to N$ there exists a unique up to homotopy morphism of DG $A$-module $\beta: P\to M$ such that $\alpha=\gamma\beta$. Equivalently, $P^\natural$ is a projective graded $A^\natural$-module and $\Hom_A(P,-)$ preserves quasi-isomorphisms.  \end{chunk}

\begin{chunk}\label{sres}
  A \emph{semiprojective resolution} of a DG $A$-module  $M$  is a surjective quasi-isomorphism of DG $A$-modules $\e: P\to M$ where  $P$ is a semiprojective DG $A$-module.  Semiprojective resolutions exist and any two semiprojective resolutions of $M$ are unique up to  homotopy equivalence.  \end{chunk}

%\begin{chunk}
%We define $M\otimes_A N$ to be the DG $A$-module with \begin{align*}(M\otimes_A N)^\natural:&=M\otimes_Q N/Q\{am\otimes n-(-1)^{|a||m|}m\otimes an: a\in A, m\in M, n\in N\} \\ \del^{M\otimes_A N}:&=\del^M\otimes N+M\otimes \del^N,\text{ and }\\ a\cdot m\otimes n:&=am\otimes n=(-1)^{|a||m|}m\otimes an.\end{align*} We remark that $M\otimes_A N$ is a quotient complex of $M\otimes_Q N$. 
%\end{chunk}

%  \begin{chunk} \label{ext2}
  %For DG $A$-modules $M$ and $N$, define $$\Ext_A^*(M,N):=\h(\Hom_A(P,N))\text{ and }\Tor^A_*(M,N):=\h(P\otimes_A N)$$ where $P$ is a semiprojective resolution of $M$ over $A$. Since any two semiprojective resolutions of $M$ are homotopy equivalent, $\Ext_A^*(M,N)$ and $\Tor^A_*(M,N)$ independent of choice of $M$. 
  %\end{chunk}

%%%%%%%%%%%%%%%%%%%%%%%%%%%%%%%%%%%%%%%%%%%%%%%%

%%%%%%%%%%%%%%%%%%%%%%%%%%%%%%%%%%%%%%%%%%%%%%%%

\section{Universal Resolutions}\label{ur}

Fix a commutative ring $Q$. Let $\f=f_1,\ldots, f_n$ be a list of elements in  $Q$. Let $$E:=Q\langle \xi_1,\ldots,\xi_n|\del\xi_i=f_i\rangle$$
be  the Koszul complex on $\f$ over $Q$. That is, $E$ is the DG $Q$-algebra with $E^\natural$ the exterior algebra on a free $Q$-module with basis $\xi_1,\ldots,\xi_n$ of homological degree 1, and differential $\del\xi_i=f_i$.  

Set $R:=Q/(\f)$. Via the augmentation map $E\to R$, every complex of $R$-modules is a DG $E$-module. 
Finally, let $\Poly:=R[\chi_1,\ldots, \chi_n]$ be a graded polynomial ring where each $\chi_i$ has homological degree -2. 

We will also need to refer to the graded $R$-linear dual of $\Poly$ throughout  Section \ref{ur}. Let $\Gamma$ denote the graded $R$-linear dual of $\Poly$ and  let  $\{y^{(H)}\}_{H\in \N^n}$ be the $R$-basis  of ${\Gamma}$ dual to  $\{\chi^H:=\chi_1^{h_1}\ldots \chi_n^{h_n}\}_{H\in \N^n}$  the standard $R$-basis  of $\Poly$. Then ${\Gamma}$ is a graded  $\Poly$-module via the action $$\chi_i\cdot y^{(H)}:=\left\{\begin{array}{cl} y^{(h_1,\ldots,h_{i-1},h_i-1,h_{i+1},\ldots,h_n)}  & h_i\geq 1 \\
0 & h_i=0 \end{array}\right.$$

\begin{chunk}
Let $M$ be a DG $E$-module. We let $\lambda_i^M$ denote left multiplication by $\xi_i$ on $M$, and when $M$ is clear from context we simply write $\lambda_i$. As $E$ is graded-commutative it follows that $\lambda_i\in\Hom_E(M,M)_1.$ Moreover, $\lambda_i$ is a null-homotopy for $f\id^M$. 
\end{chunk}

\begin{chunk}\label{c6}
Let $M$ be a DG $E$-module.   Define $U_E(M)$ to be the DG $\Poly\otimes_Q E$-module with $$U_E(M)^\natural \cong (\Gamma\otimes_Q M)^{\natural}$$ and differential given by the formula $$\del=1\otimes \del^{M}+\sum_{i=1}^n  \chi_i\otimes \lambda_i.$$ It is a straightforward check that $U_E(M)$ is a DG $\Poly\otimes_QE$-module. 
\end{chunk}

\begin{chunk}
We say that $\f$ is  \emph{Koszul-regular} if  the augmentation map $E\to R$ is a quasi-isomorphism. Equivalently, $$\hh_i(E)=\left\{\begin{array}{cl} R & i=0\\ 0 & i\neq 0 \end{array}\right.$$ If $\f$ is a $Q$-regular sequence, then it is Koszul-regular.  When $Q$ is local and $\f$ is contained in the maximal ideal of $Q$, the converse holds. 
\end{chunk}

\begin{chunk}\label{c5}
Assume that $\f$ is a Koszul-regular sequence and fix a complex of $R$-modules $M$.  Let $\e: F\xra{\simeq} M$ be a semiprojective resolution of $M$ over $E$.  By \cite[2.4]{CD2}, $U_E(F)\to M$ is  a semiprojective resolution over $R$ where the augmentation map is given by 
  $$a\otimes y^{(H)}\otimes x\mapsto \left\{\begin{array}{cl}  a\e(x) & |H|=0 \\ 0 & |H|>1 \end{array}\right.$$ In particular, for any complex of $R$-modules $N$ \begin{align*}\Ext_R^*(M,N)&\cong \h(\Hom_R(U_E(F),N)) \\\Tor_*^R(M,N)&\cong \h(U_E(F)\otimes_R N)\end{align*}  are graded $\Poly$-modules. 
 When we refer to $\Ext_R^*(M,N)$ and $\Tor_*^R(M,N)$ as  graded $\Poly$-modules, we are considering the $\Poly$-module structures determined by the isomorphisms above. 
 Moreover, the $\Poly$-module structure on $\Ext_R^*(M,N)$ or  $\Tor^R_*(M,N)$ is independent of choice $F$ and is natural in both $M$ and $N$ (see \cite[3.1]{CD2}) . 
\end{chunk}

%%%%%%%%%%%%%%%%%%%%%%%%%%%%%%%%%%%%%%%%%%%%%%%%%%%

%%%%%%%%%%%%%%%%%%%%%%%%%%%%%%%%%%%%%%%%%%%%%%%%%%%%%%%
\section{Perturbations}\label{sP}

Fix a DG algebra $A$ over a commutative ring $Q$.  Fix a DG $A$-module $X$. Recall that for  $\alpha,\beta \in \Hom_A(X,X)$ the commutator of $\alpha$ and $\beta$ is   $$[\alpha,\beta]:=\alpha\beta-(-1)^{|\alpha||\beta|}\beta\alpha.$$

\begin{chunk}\label{central}
Let $\gamma\in \Hom_A(X,X)$. We say that $\gamma$ is a \emph{central map} if for each $\sigma\in \Hom_A(X,X)$, $$\gamma \sigma=(-1)^{|\sigma||\gamma|}\sigma\gamma.$$ That is, $[\gamma,-]=0$ on $\Hom_A(X,X)$.  
\end{chunk}

\begin{chunk}
For $\alpha\in \Hom_A(X,X)$, $$\del^{\Hom_A(X,X)}(\alpha)=[\del^X,\alpha].$$ In particular, $[\del^X,\alpha]=0$ if and only if $\alpha \in \zz(\Hom_A(X,X)).$ 
\end{chunk}

\begin{chunk}\label{cper}
 Assume $\delta\in \zz_{-1}(\Hom_A(X,X))$ satisfies $\delta^2=0$. Define $X^\delta$ to be the perturbation of $X$ by $\delta$. That is, $X^\delta$ is the DG $A$-module where  $(X^\delta)^\natural$ is $X^\natural$ as a graded $A$-module and differential $\del^X+\delta.$
\end{chunk}

\begin{chunk}\label{c1}
For $H=(h_1,\ldots,h_n)\in \N^n$, 
we let  $$|H|:=h_1+\ldots+ h_n.$$ Also, define
$$H_i:=\left\{\begin{array}{cl} (h_1,\ldots,h_{i-1},h_i-1,h_{i+1},\ldots,h_n)  & h_i\geq 1 \\
0 & h_i=0 \end{array}\right.$$ and $$H_{i,j}:=(H_i)_j=(H_j)_i.$$\end{chunk}

\begin{chunk}\label{stronghom}
Fix $d\in \Z$. Let $\bm{\alpha}=\alpha_1,\ldots,\alpha_n$ and $\bm{\beta}=\beta_1,\ldots,\beta_n$ be sequences of elements in $\Hom_A(X,X)_d$.    We say that $\bm{\alpha}$ and $\bm{\beta}$ are
\emph{homotopic}, denoted $\bm{\alpha}\sim \bm{\beta}$, provided that $$\alpha_i\sim \beta_i$$ for each $1\leq i \leq n$. 

 We say that $\bm{\alpha}$ and $\bm{\beta}$ are
 \emph{strongly homotopic}, denoted $\bm{\alpha}\approx \bm{\beta}$, if there exists a family of maps $\bm{\tau}:=\{\tau^{(H)}\}_{H\in \N^n}$ in $\Hom_A(X,X)$ satisfying the following:
\begin{enumerate}
\item $|\tau^{(H)}|=|H|(d+1)$,
\item $\tau^{(\bm{0})}=\id^X$, and 
\item $[\del^X,\tau^{(H)}]=\sum_{i=1}^n\tau^{(H_i)}\alpha_i-\beta_i\tau^{(H_i)}$ for each $|H|>0$. 
\end{enumerate} We say that $\bm{\tau}$ is a \emph{system of higher strong homotopies from $\bm{\alpha}$ to $\bm{\beta}$}.  
\end{chunk}

\begin{remark}
Suppose  $\bm{\tau}=\{\tau^{(H)}\}_{H\in \N^n}$ is a system of higher strong homotopies from $\bm{\alpha}$ to $\bm{\beta}$.  It is immediate that   $\bm{\alpha}$ and  $\bm{\beta}$ are homotopic, as $\tau^{(\bm{e}_i)}$ is a homotopy for $\alpha_i$ and $\beta_i$ where $\bm{e}_i$ is the $n$-tuple with a 1 in the $i^{\text{th}}$ spot and 0's everywhere else. 
\end{remark}

\begin{proposition}\label{p1}
Let $X$ be a DG $A$-module and suppose  $\bm{\alpha}=\alpha_1,\ldots,\alpha_n$ and $\bm{\beta}=\beta_1,\ldots,\beta_n$  are homotopic sequences of central maps on $X$ of positive odd degree $d$. If $\h_{i}(\Hom_A(X,X))=0$ for all $i\geq d$, then $\bm{\alpha}$ and  $\bm{\beta}$ are strongly homotopic. \end{proposition}
\begin{proof}
We set $\tau^{(\bm{0})}:=\id^X$ and construct the rest by induction. Suppose $|H|>0$ and assume we've constructed $\{\tau^{(H')}: |H'|<|H|\}$ satisfying (1)-(3).   
For each $i$, notice that $$[\del^X, \tau^{(H_i)}(\alpha_i-\beta_i)]=[\del^X, \tau^{(H_i)}](\alpha_i-\beta_i)+\tau^{(H_i)}[\del^X,\alpha_i-\beta_i].$$  As $\alpha_i\sim\beta_i$, it follows that $[\del^X,\alpha_i-\beta_i]=0$ and hence,  $$\tau^{(H_i)}[\del^X,\alpha_i-\beta_i]=0.$$ Moreover, \begin{align*} [\del^X, \tau^{(H_i)}](\alpha_i-\beta_i)&=\left(\sum_{j=1}^n \tau^{(H_{i,j})}\alpha_j-\beta_j \tau^{(H_{i,j})}\right)(\alpha_i-\beta_i) \\
&=\sum_{j=1}^n \tau^{(H_{i,j})}(\alpha_j-\beta_j )(\alpha_i-\beta_i) 
\end{align*}
Now using that  $\alpha_i$ and $\beta_i$ are central maps of the same odd degree it follows that $$\sum_{i=1}^n\sum_{j= 1}^n \tau^{(H_{i,j})}(\alpha_j-\beta_j )(\alpha_i-\beta_i)=0.$$ Thus, $$\left[\del^X,\sum_{i=1}^n\tau^{(H_i)}(\alpha_i-\beta_i)\right]=0.$$ and hence, 
$$\sum_{i=1}^n\tau^{(H_i)}(\alpha_i-\beta_i)\in \zz(\Hom_A(X,X)).$$  As  $\h_{i}(\Hom_A(X,X))=0$ for all $i\geq d$ and $$\left|\sum_{i=1}^n\tau^{(H_i)}(\alpha_i-\beta_i)\right|=|H_i|(d+1)+d\geq d,$$ it follows that  there exists  a map $\tau^{(H)}$ in $\Hom_{A}(X,X)$ of degree $$|H_i|(d+1)+d+1=|H|(d+1)$$ satisfying $$[\del^X,\tau^{(H)}]=\sum_{i=1}^n\tau^{(H_i)}(\alpha_i-\beta_i).$$ Finally, as $\beta_i$ is central for each $i$, condition (3) in \ref{stronghom} is satisfied. Thus, by induction we are done.  
\end{proof}

\begin{definition}\label{pss}
Let  $\bm{\alpha}=\alpha_1,\ldots,\alpha_n$ and $\bm{\beta}=\beta_1,\ldots,\beta_n$ be sequences of  maps of positive odd degree $d$ on $X$, 
$\bm{\gamma}=\gamma_1,\ldots, \gamma_n$ is a sequence of centrals maps on $X$  of degree $|\gamma_i|=-d-1$, and set $$\delta:=\sum_{i=1}^n\gamma_i\alpha_i  \ \text{ and } \ \e:=\sum_{i=1}^n\gamma_i \beta_i.$$ We say that $(\bm{\alpha},\bm{\beta}, \bm{\gamma},\bm{\tau})$ is a \emph{perturbing system on X} if the following hold 
\begin{enumerate}
\item $\bm{\tau}:=\{\tau^{(H)}\}_{H\in \N^n}$  is a system of higher strong homotopies from  $\bm{\alpha}$ to  $\bm{\beta}$,
\item $\delta,\e\in \zz_{-1}(\Hom_A(X,X))$ 
\item $\delta^2=0$ and $\e^2=0$
\item For each $x\in X$, $\gamma^H(x)=0$ for all $|H|\gg 0$ where $\gamma^{H}=\gamma_1^{h_1}\ldots\gamma_n^{h_n}$\end{enumerate}
\end{definition}

\begin{Construction}\label{const1}Let  $(\bm{\alpha},\bm{\beta}, \bm{\gamma},\bm{\tau})$ be a perturbing system on $X$. 
Using  conditions (2) and (3) from Definition \ref{pss}, $X^{\delta}$ and $X^\e$ are  well-defined DG $A$-modules (see \ref{cper}). Define  $\bm{\gamma\tau}: X^\delta \to X^\e$ by $$x\mapsto \sum_{H\in \N^n} \gamma^H\tau^{(H)}(x).$$ As each  $\gamma_i$ is a central map, $$\gamma^H\tau^{(H)}(x)=\tau^{(H)}\gamma^H(x)$$ for each $H\in \N^n$.  Hence, condition (4) implies that $$ \sum_{H\in \N^n} \gamma^H\tau^{(H)}(x)$$ is a finite sum. Moreover, $\bm{\gamma\tau}$ is $A$-linear. Since $|\gamma_i|=-d-1$ it follows that $$|\gamma^H\tau^{(H)}|=|\gamma^H|+|\tau^{(H)}|=|H|(-d-1)+H(d+1)=0$$  and so  $\bm{\gamma\tau}$ is a degree 0 map in $\Hom_A(X,X)$. 

Finally, we claim that $\bm{\gamma\tau}$ is a morphism of DG $A$-modules. The proof is given with the following string of equalities.  Adopting the convention  $\gamma^{\bm{0}}=\id^X$, the fifth equality below follows from \ref{stronghom}(3)
\begin{align*}(\del^X+\e)\bm{\gamma\tau}&=(\del^X+\e)\sum_{H\in \N^n} \gamma^H\tau^{(H)}\\
&=\sum_{H\in \N^n} \gamma^H\del^X\tau^{(H)}+\sum_{i=1}^n\sum_{H\in \N^n}\gamma_i\gamma^{H}\beta_i\tau^{(H)}\\
&=\del^X+\sum_{|H|>0} \gamma^H\del^X\tau^{(H)}+\sum_{|H|>0}\sum_{i=1}^n\gamma^{H}\beta_i\tau^{(H_i)}\\
&=\del^X+\sum_{|H|>0} \left(\gamma^H\del^X\tau^{(H)}+\sum_{i=1}^n\gamma^H\beta_i\tau^{(H_i)}\right)\\
&=\del^X+\sum_{|H|>0} \left(\gamma^H\tau^{(H)}\del^X+\sum_{i=1}^n\gamma^H\tau^{(H_i)}\alpha_i\right)\\
&=\del^X+\sum_{|H|>0} \gamma^H\tau^{(H)}\del^X+\sum_{H\in \N^n}\sum_{i=1}^n\gamma_i\gamma^{H}\tau^{(H)}\alpha_i\\
&=\sum_{H\in \N^n} \gamma^H\tau^{(H)}\del^X+\sum_{H\in \N^n}\gamma^{H}\tau^{(H)}\sum_{i=1}^n\gamma_i\alpha_i\\
&=\bm{\gamma\tau}(\del^X+\delta).
\end{align*} Thus, $\bm{\gamma\tau}$ is a morphism of DG $A$-modules. 
\end{Construction}

\begin{theorem}\label{t1}
Let $X$ be a DG $A$-module and suppose $(\bm{\alpha},\bm{\beta}, \bm{\gamma},\bm{\tau})$ is a perturbing system on $X$. Using the notation from Construction \ref{const1}, the morphism of DG $A$-modules $\bm{\gamma\tau}: X^{\delta}\to X^{\e}$  is an isomorphism.   \end{theorem}
\begin{proof}
For $n\in \N$, define $$X(n):=\{x\in X^{\delta}: \gamma^H(x)=0\text{ for all }|H|>n\}.$$  Since each $\gamma_i$ is central, it follows that $X(n)$ is a DG $A$-submodule of $X^\delta$. Indeed, for $x\in X(n)$ and $|H|>n$ we have that $$\gamma^H(\del^X+\delta)(x)=(\del^X+\delta)\gamma^H(x)=0.$$  Similarly, $Y(n)$ is a DG $A$-submodule of $X^{\e}$ where $Y(n)^\natural=X(n)^\natural$. Again using that $\gamma_i$ is central for each $i$,  for each $H\in \N^n$ it follows that $$\gamma^H\circ\bm{\gamma\tau}=\bm{\gamma\tau}\circ \gamma^H.$$ Thus, $\bm{\gamma\tau}(X(n))\con Y(n).$ Set $\bm{\gamma\tau}(n):=\bm{\gamma\tau}|_{X(n)}: X(n)\to Y(n)$. 

We have a chain of DG $A$-submodules $$0=X(0)\con X(1)\con X(2) \ldots \ \  \text{ and } \ \ 0=Y(0)\con Y(1)\con Y(2) \ldots $$  of $ X^{\delta}$ and $X^{\e}$, respectively. Moreover, by assuming condition (3) it follows that \begin{equation}\label{eqs}\varinjlim X(n)=X^{\delta}, \ \varinjlim Y(n)=X^{\e}, \  \text{ and } \  \varinjlim \bm{\gamma\tau}(n)=\bm{\gamma\tau}.\end{equation} Finally, it is clear that $\bm{\gamma\tau}(0)$ is an isomorphism  and that $\bm{\gamma\tau}(n)$ induces an isomorphism $$X(n)/X(n-1)\to Y(n)/Y(n-1)$$ for each $n\in \N$. Hence, by induction $\bm{\gamma\tau}(n)$ is an isomorphism for each $n \in \N$.  Thus, by (\ref{eqs}) it follows that $\bm{\gamma\tau}$ is an isomorphism. 
\end{proof}

%%%%%%%%%%%%%%%%%%%%%%%%%%%%%%%%%%%%%%%%%%%%%%%%

%%%%%%%%%%%%%%%%%%%%%%%%%%%%%%%%%%%%%%%%%%%%%%%%

\section{Equivariant Isomorphisms}

\begin{chunk}\label{equivariant}
Let $\vp: A\to A'$ be a morphism of DG $Q$-algebras. Suppose $M$ is a DG $A$-module and $M'$ is a DG $A'$-module. A morphism of complexes $\psi: M\to M'$ is $\vp$-\emph{equivariant} if $$\psi(am)=\vp(a)\psi(m)$$ for all $a\in A$ and $m\in M$. 
Similarly, a morphism of complexes $\psi': M'\to M$ is $\vp$-\emph{equivariant} if $$\psi'(\vp(a)m')=a\psi'(m')$$ for all $a\in A$ and $m'\in M$.

Let $\psi$ be a $\vp$-equivariant map. Then $\psi$ is a morphism of DG $A$-modules when $M'$ is regarded as an $A$-module via restriction of scalars along $\vp$.
 If $\psi$ is an isomorphism of complexes, we say that $\psi$ is a $\vp$-\emph{equivariant isomorphism}. 
\end{chunk}

  First, we fix some notation which will be used throughout the  section.

\begin{notation}\label{n1}
 Let  $Q$ be a commutative ring and fix two lists of elements    $\f=f_1,\ldots, f_n$ and $\f'=f_1'\ldots, f_n'$ in $Q$. Set 
 \begin{align*}
 E&:= Q\langle \xi_1,\ldots,\xi_n|\del\xi_i=f_i\rangle  &  &E':=Q\langle \xi_1',\ldots, \xi_n'|\del \xi_i'=f_i'\rangle \\
  \Poly&:=Q/(\f)[\chi_1,\ldots,\chi_n]   & &\Poly':=Q/(\f')[\chi_1',\ldots,\chi_n'] \\
   \Gamma&:=\Hom_{Q/(\f)}(\Poly,Q/(\f))    & &  \Gamma':=\Hom_{Q/(\f')}(\Poly',Q/(\f')) 
 \end{align*}
 where each  $\chi_i$ and $\chi'_i$ have homological  degree $-2$. Also, let   $\Lambda:=E\otimes_QE'$.  For a DG $\Lambda$-module $X$ we let $\lambda_i$ and $\lambda_i'$ denote left multiplication  on $X$  by $\xi_i$ and $\xi_i'$, respectively. 
 
Suppose that   $I$ is an ideal of $Q$ containing $(\f,\f')$,  set $R=Q/I$ and 
  \begin{align*}
  \pp&:=R\otimes_{Q}\Poly=R[\chi_1,\ldots,\chi_n]   & &\pp':=R\otimes_Q\Poly' =R[\chi_1',\ldots,\chi_n'] \\
   \mathcal{D}&:=\Hom_R(\pp,R)   & &  \mathcal{D}':=\Hom_{R}(\pp',R) 
 \end{align*}
  \end{notation}
 
% $$E:=Q\langle \xi_1,\ldots,\xi_n|\del\xi_i=f_i\rangle$$ and   $\ov{Q}:=Q/(\f)$. For a DG $E$-module $X$, we let $\lambda_i^X$ denote left multiplication by $\xi_i$ and we simply write $\lambda_i$ when $X$ is clear from context. Finally, let $\Poly=\ov{Q}[\chi_1,\ldots,\chi_n]$ be  a polynomial ring where each $\chi_i$ has  degree $-2$, $\Gamma$ is the $\ov{Q}$-linear dual of $\Poly$. 
 
% We similarly define $E'$, $\ov{Q}'$, $\Poly'$, $\lambda_i'$, and $\Gamma'$. 
 
% Finally, $\Lambda:=E\otimes_QE'$. Suppose that   $I$ is an ideal of $Q$ containing $(\f,\f')$ and set $R=Q/I$. Let $\pp=R[\chi_1,\ldots,\chi_n]$ and $\pp'=R[\chi_1',\ldots, \chi_n']$ be   polynomial rings where each $\chi_i$ and $\chi_i'$ have  degree $-2$ and $\mathcal{D}$ and $\mathcal{D}'$  is the $R$-linear dual of $\pp$ and $\pp'$, respectively. 

\begin{remark}
\label{r1}
Set $\ov{Q}:=Q/(\f)$ and $\ov{Q}':=Q/(\f')$ and assume that $\f$ and $\f'$ are Koszul-regular sequences. For  complexes of $R$-modules $M$ and $N$, $$\Ext_{\ov{Q}}^*(M,N)=\h(\Hom_{\ov{Q}}(U_{E}(F), N))$$ is a graded module over $\Poly$ where $F$ is a semiprojective resolution of $M$ over $E$ (see \ref{c5}). As $I$ annihilates $\Ext_{\ov{Q}}^*(M,N)$, $\Ext_{\ov{Q}}^*(M,N)$ is a graded $\pp$-module. 
Moreover,  since $N$ is an $R$-module, using adjunction we have the following isomorphism of graded $\pp$-modules \begin{equation}\Ext_{\ov{Q}}^*(M,N)\cong \h(\Hom_R(R\otimes_{\ov{Q}} U_{E}(F),N)). \label{iso1}
\end{equation}
Similarly, there is an isomorphism of graded $\pp'$-modules 
\begin{equation}\Ext_{\ov{Q}'}^*(M,N)\cong \h(\Hom_R(R\otimes_{\ov{Q}'} U_{E'}(F'),N)) \label{iso2}
\end{equation} where $F'$ is a semiprojective resolution of $M$ over $E'$. 
There are analogous statements for  $\Tor_*^{\ov{Q}}(M,N)$ and $\Tor^{\ov{Q}'}(M,N)$ as graded $\pp$- and $\pp'$-modules, respectively. 
\end{remark}

In this section, first,  we will be   comparing $\Ext_{\ov{Q}}^*(M,N)$ with  $\Ext_{\ov{Q}'}^*(M,N)$ and $\Tor^{\ov{Q}}(M,N)$ with $\Tor^{\ov{Q}'}(M,N)$ for   $R$-modules $M$ and $N$ when $\f$ and $\f'$ are assumed to be Koszul-regular sequences. Roughly speaking,  if each $\lambda_i$  and $\lambda_i'$ act similarly on a semiprojective resolution of $M$ over $\Lambda$, then we have isomorphisms between the corresponding Ext and Tor modules. Later in the section, we will be interested in the case when $\f$ is Koszul-regular and $\f'$ is the zero sequence.

\begin{lemma}\label{l3}
 For any DG $E$-module $Y$ and DG $E'$-module $Y'$, $$U_E(Y)= (\Gamma\otimes_{Q} Y )^\delta \ \text{ and } \ U_{E'}(Y')= (\Gamma'\otimes_{Q} Y' )^{\delta'}$$ where     $\delta:=\sum_{i=1}^n \chi_i\otimes  \lambda_i$ and   $\delta':=\sum_{i=1}^n \chi_i'\otimes  \lambda_i'. $ 
\end{lemma}
\begin{proof}
This follows immediately by examining the underlying graded modules and the differentials (see \ref{c6} and \ref{cper}). \end{proof}

\begin{chunk}\label{c10}
Let $M$ be an $R$-module. Suppose that $f_i-f_i'\in I \ann_Q(M)$ for each $i$, that is, $$f_i-f_i'=\sum_j x_{i,j}g_{i,j}$$ where $x_{i,j}\in I$ and $g_{i,j}\in \ann_Q(M)$.  We set $B:=\Lambda\langle \zeta_{i,j}|\del \zeta_{i,j}=g_{i,j}\rangle$ and as $f_i,f_i', g_{i,j}\in \ann_Q(M)$ it follows that $M$ is a DG $B$-module. 
A bounded below semiprojective resolution of $M$ over $B$ is called a \emph{strong $\Lambda$-resolution of $M$}. 
Since $B$ is semiprojective over $\Lambda$,  a strong $\Lambda$-resolution of $M$ is a semiprojective resolution of $M$ over $\Lambda$. 

By \ref{semiproj}, strong $\Lambda$-resolutions of $M$ exist. Moreover, any two strong $\Lambda$-resolutions of $M$ are homotopy equivalent DG $\Lambda$-modules (c.f. \ref{sres}). 
\end{chunk}

 Lemma \ref{l4} was inspired by, and generalizes, Claim 2 in the proof of  \cite[2.1]{AvI}.  The proof of Lemma \ref{l4} relies heavily on the ideas in Section \ref{sP}. 
\begin{lemma}
\label{l4}Let $M$ be an $R$-module.  Suppose that $f_i-f_i'\in I \ann_Q(M)$ for each $i$. 
For any strong $\Lambda$-resolution $F\xra{\simeq} M$, there is an isomorphism of DG 
$\pp\otimes_Q \Lambda$-modules 
 $$(\mathcal{D}\otimes_Q  F)^\delta \cong (\mathcal{D}\otimes_Q  F)^{\e},$$ where    $\delta=\sum_{i=1}^n \chi_i\otimes \lambda_i$ and  $\e=\sum_{i=1}^n \chi_i\otimes \lambda_i'$.
\end{lemma}
\begin{proof}

We let $\sigma_{i,j}\in \Hom_B(F,F)_1$ denote left multiplication by $\zeta_{i,j}$. 
By assumption, $$[\del^F,\lambda_i-\lambda_i']=\left[\del^F, \sum_j x_{i,j}\sigma_{i,j}\right]$$  and hence,  $$\lambda_i-\lambda_i'-\sum_jx_{i,j}\sigma_{i,j}\in \zz_1(\Hom_B(F,F)).$$ 
As $F\xra{\simeq} M$ and $F$ is semiprojective over $B$, it follows that $$\Hom_B(F,F)\xra{\simeq} \Hom_B(F,M).$$  Since $F$ is nonnegatively graded and $M$ is concentrated in degree $0$,  $$\Hom_B(F,M)_1=0$$ and hence,  $$0=\h_1(\Hom_B(F,M))=\h_1(\Hom_B(F,F)).$$

Thus, $\zz_1(\Hom_B(F,F))=\bb_1(\Hom_B(F,F))$
and so there exists $\tau\in \Hom_B(F,F)_2$ such that $$[\del^F,\tau]=\lambda_i-\lambda_i'-\sum_jx_{i,j}\sigma_{i,j}.$$
Thus, $$\lambda_i-\sum_jx_{i,j}\sigma_{i,j}\sim \lambda_i'$$ for all $i$. Since  $\lambda_i,$ $\lambda_i'$ and $\sigma_i'$ are central on $F$ over $B$, we apply Proposition \ref{p1} to conclude that  
$$\tilde{\bm{\alpha}}\approx \tilde{\bm{\beta}}$$ where 
$$\tilde{\alpha}_i:=\lambda_i-\sum_jx_{i,j}\sigma_{i,j} \text{ and }\tilde{\beta}_i:=\lambda_i'$$
for each $1\leq i \leq n$. 
 Also,  as maps on $\mathcal{D}\otimes_Q F$,
$$1\otimes\lambda_i= 1\otimes\lambda_i-1\otimes \sum_jx_{i,j}\sigma_{i,j}$$ since each $x_{i,j}$ is in  $I$ which annihilates $\mathcal{D}$. Therefore, the sequences $(1\otimes \lambda_i)_{i=1}^n$ and $\tilde{\bm{\beta}}$ are  strongly homotopic maps of degree 1 on $\mathcal{D}\otimes_Q F$.

In the notation of Section \ref{sP}, set \begin{align*}
A&=\pp \otimes_Q B \\
X&= \mathcal{D}\otimes_Q F\\
\bm{\alpha}&=1\otimes \lambda_1,\ldots, 1\otimes \lambda_n\\ 
\bm{\beta}&=1\otimes \lambda_1',\ldots, 1\otimes \lambda_n'\\
\bm{\gamma}&=\chi_1\otimes 1,\ldots, \chi_n\otimes 1
\end{align*} and let $\bm{\tau}$ denote the strong homotopy from $\bm{\alpha}$ to $\bm{\beta}$ (the existence of $\bm{\tau}$ was justified above). 
It is easily checked that $(\bm{\alpha},\bm{\beta},\bm{\gamma},\bm{\tau})$ is a perturbing system on $X$ (see Definition \ref{pss}).  Therefore, Theorem \ref{t1} yields an isomorphism of DG $\pp\otimes_Q B$-modules  $$(\mathcal{D}\otimes_Q  F)^\delta \cong (\mathcal{D}\otimes_Q  F)^{\e},$$ where recall that  $\delta=\sum_{i=1}^n \chi_i\otimes \lambda_i$ and  $\e=\sum_{i=1}^n \chi_i\otimes \lambda_i'$. 
\end{proof}

We are now equipped to prove the main results of this paper. Theorem \ref{t2} generalizes \cite[2.1(2)]{AvI}, as discussed in the introduction.

\begin{theorem}\label{t2}
We adopt the notation from Notation \ref{n1}. Further assume that $\f$ and $\f'$ are Koszul-regular sequences. Let $\vp: \pp\to \pp'$ be the  morphism of DG $R$-algebras given by $$\chi_i\mapsto \chi_i'. $$
 Assume that $M$ is an $R$-module such that for each $i$ $$f_i-f_i'\in I \ann_Q(M).$$
For each complex of $R$-modules $N$,  we have the following $\vp$-equivariant isomorphisms:
\begin{enumerate}
\item $\Ext_{Q/(\f)}^*(M,N)\cong \Ext_{Q/(\f')}^*(M,N)$
\item $\Tor^{Q/(\f)}_*(M,N)\cong \Tor^{Q/(\f')}_*(M,N)$
\end{enumerate}
\end{theorem}
\begin{proof}
The proof of (2) is similar to the proof of (1). So we only prove (1) and leave  translating the proof of (1) to the context of (2) as an  exercise for the reader.

By \ref{c10} and Lemma \ref{l4}, there exists a semiprojective resolution $F\xra{\simeq} M$ over $\Lambda$ and a canonical isomorphism of DG $\pp\otimes_Q \Lambda$-modules 
 \begin{equation}\label{eq10}(\mathcal{D}\otimes_Q  F)^\delta \cong (\mathcal{D}\otimes_Q  F)^{\e},\end{equation} where    $\delta=\sum_{i=1}^n \chi_i\otimes \lambda_i$ and  $\e=\sum_{i=1}^n \chi_i\otimes \lambda_i'$. Moreover, $\vp$ induces the $\vp\otimes \Lambda$-equivariant isomorphism  \begin{equation}\label{eq11}(\mathcal{D}'\otimes_Q  F)^{\delta'}\xra{\Hom_R(\vp,R)\otimes \Lambda} (\mathcal{D}\otimes_Q  F)^{\e} \end{equation} where $\delta'=\sum_{i=1}^n \chi_i'\otimes \lambda_i'.$ Composing the isomorphisms in (\ref{eq10}) and (\ref{eq11}), we obtain a $\vp\otimes \Lambda$-equivariant isomorphism  \begin{equation}\label{eq12}(\mathcal{D}\otimes_Q  F)^\delta\cong    (\mathcal{D}'\otimes_Q  F)^{\delta'}.\end{equation}

Using that $\f$ and $\f'$ are Koszul-regular and  Lemma \ref{l3}, we  obtain  the following isomorphism of DG  $\pp\otimes_Q \Lambda$-modules and  isomorphism of DG $\pp'\otimes_Q \Lambda$-modules
$$
R\otimes_{Q/(\f)} U_{E}(F) \cong  (\mathcal{D}\otimes_Q  F)^\delta \ \text{ and }  \ 
R\otimes_{Q'/(\f')} U_{E'}(F)  \cong  (\mathcal{D}'\otimes_Q  F)^{\delta'},$$
 respectively. Therefore, we have a  $\vp\otimes \Lambda$-equivariant isomorphism $$R\otimes_{Q/(\f)} U_{E}(F){\cong }R\otimes_{Q/(\f')} U_{E'}(F)$$ using these isomorphisms and  (\ref{eq12}). 
 Applying $\Hom_R(-,N)$ providues us with  a $\vp\otimes \Lambda$-equivariant isomorphism $$\Hom_R(R\otimes_{Q/(\f)} U_{E}(F), N)\cong \Hom_R(R\otimes_{Q/(\f')} U_{E'}(F), N).$$ Now using (\ref{iso1}) and(\ref{iso2}) from Remark \ref{r1},  we obtain (1). \end{proof}

%First we prove part (1). As $\Lambda$ is semiprojective over $E$ and $E'$, it follows that $F\xra{\simeq} M$ is a semiprojective resolution over $E$ and $E'$. 
%Since $\f$ is Koszul-regular, $U_{E}(F)\xra{\simeq}M$ is a semiprojective resolution of $M$ over  $\ov{Q}$ (see \ref{c5}). Hence, $$\hh(\Hom_{\ov{Q}}(U_E(F), N))=\Ext_{\ov{Q}}^*(M,N)$$ and similarly, $$\hh(\Hom_{\ov{Q}'}(U_{E'}(F), N))=\Ext_{\ov{Q}'}^*(M,N).$$
%Since $N$ is a complex of $R$-modules, we have  isomorphisms of DG $\pp\otimes_Q\Lambda$-modules \begin{align*}\Hom_{\ov{Q}}(U_{E}(F), N)&\cong \Hom_R(R\otimes_{\ov{Q}} U_{E}(F),N) \\
%\Hom_{\ov{Q}'}(U_{E'}(F), N)&\cong \Hom_R(R\otimes_{\ov{Q}'} U_{E'}(F),N) \\
%\end{align*} 
%Therefore, it suffices to show that we have an isomorphism of DG $\pp\otimes_Q \Lambda$-modules \begin{equation}\label{crucialiso}R\otimes_{\ov{Q}} U_{E}(F)\cong R\otimes_{\ov{Q}'} U_{E'}(F).\end{equation}
% By Lemma \ref{l3}, we have isomorphisms of DG  $\pp\otimes_Q \Lambda$-modules
%\begin{align}\label{eq5}
%R\otimes_{\ov{Q}} U_{E}(F)& \cong  (\mathcal{D}\otimes_Q  F)^\delta  \\
%R\otimes_{\ov{Q'}} U_{E'}(F) & \cong  (\mathcal{D}\otimes_Q  F)^{\delta'}. \label{eq6}
%\end{align} Also, by (\ref{eq10}), we have a canonical isomorphism of DG $\pp\otimes_Q \Lambda$-modules 
 %$$(\mathcal{D}\otimes_Q  F)^\delta \cong (\mathcal{D}\otimes_Q  F)^{\delta'}.$$ Combining these with 
 %(\ref{eq5}) and (\ref{eq6}) we obtain $(\ref{crucialiso})$ which finishes the proof of (1). 

\begin{chunk}\label{basspoin}
Let $Q$ be local and $M$ a $Q$-module. We let $\nu(M)$ denote the minimal number of generators of $M$. For a complex of $Q$-modules $N$,    the \emph{Bass series of $M$ and $N$} is 
$$\Bass_{Q}^{M,N}(t)=\sum_{i=0}^\infty \nu(\Ext_Q^i(M,N))t^i.$$  When $M=k$, this recovers the classical Bass series of $N$. 

The \emph{Poincar\'{e} series of $M$ and $N$} is $$\Poin^Q_{M,N}(t)=\sum_{i=0}^\infty \nu(\Tor^Q_i(M,N))t^i.$$ When $M=k$, this recovers the classical Poincar\'{e} series of $N$. 
\end{chunk}

From Theorem \ref{t2} we get the following immediate  corollary. 
\begin{corollary}\label{cor1}
 Further assume that $Q$ is local. 
For each complex of $R$-modules $N$,  the following hold:
\begin{enumerate}
\item $\Bass_{Q/(\f)}^{M,N}(t)=\Bass_{Q/(\f')}^{M,N}(t)$
\item $\Poin^{Q/(\f)}_{M,N}(t)=\Poin^{Q/(\f')}_{M,N}(t)$
\end{enumerate}

\end{corollary}

Our last main theorem of this article is given below in Theorem \ref{t9}. The techniques from Section \ref{sP} are again applied to establish this result. Finally, Corollary \ref{corr} is an immediate consequence of Theorem \ref{t9}. 
\begin{theorem}\label{t9}
We adopt the notation from Notation \ref{n1}. If $(\f)\con I \ann_Q(M)$ is a Koszul-regular sequence, we  have isomorphisms of graded $\pp$-modules:
\begin{enumerate}
\item $\Ext_{Q/(\f)}^*(M,N)\cong \pp\otimes_R \Ext_Q^*(M,N) $  
\item $\Tor^{Q/(\f)}_*(M,N)\cong \pp^*\otimes_R\Tor^{Q}_*(M,N)$ 
\end{enumerate}
\end{theorem}
\begin{proof}

The proof of (1) and (2) are similar, and hence, we  only show (1). 

Let $f_i'=0$ for all $1\leq i \leq n$ and 
fix a strong $\Lambda$-resolution $F\xra{\simeq} M$. By Lemma \ref{l4}, we have an isomorphism of DG $\pp\otimes_Q \Lambda$-modules 
 $$(\mathcal{D}\otimes_Q  F)^\delta \cong \mathcal{D}\otimes_Q F.$$

Therefore,  we have the following isomorphisms of graded $\pp$-modules 
\begin{align*}
\Ext_{Q/(\f)}^*(M,N)&=\hh(\Hom_R((\mathcal{D}\otimes_Q  F)^\delta,N)) \\
&\cong \hh(\Hom_R( \mathcal{D}\otimes_Q F,N)) \\
&\cong \hh(\Hom_R(\mathcal{D},R)\otimes_Q \Hom_Q(F,N)) \\
&\cong \hh(\pp\otimes_R \Hom_Q(F,N))\\
&\cong \pp \otimes_R \hh(\Hom_Q(F,N)).
\end{align*} As $F$ is semiprojective over $E$, and $E$ is semiprojective over $Q$, it follows that $F\xra{\simeq}M$ is a semiprojective resolution of $M$ over $Q$. Thus, (3) holds by the isomorphisms exhibited. 
\end{proof}

\begin{corollary}\label{corr}
 Further assume that $Q$ is local. 
\begin{enumerate}
\item $\displaystyle \Bass_{Q/(\f)}^{M,N}(t)=\dfrac{\Bass_{Q}^{M,N}(t)}{(1-t^2)^n}.$
\item $\displaystyle \Poin^{Q/(\f)}_{M,N}(t)=\dfrac{\Poin^{Q}_{M,N}(t)}{(1-t^2)^n}.$
\end{enumerate}
\end{corollary}

\end{document}